\documentclass[letterpaper, 10 pt, conference]{ieeeconf}
\usepackage{cite}
\usepackage{amsmath,amssymb,amsfonts,algorithm}
\usepackage{algorithmic}
\usepackage{graphicx,bbm}
\usepackage{mathbbol,xcolor}
\usepackage{textcomp,xspace}
\usepackage{cuted,bbm,url} 
\usepackage{tikz}
\usepackage{pgfplots}
\usepackage{pgfplotstable} 

\usepgfplotslibrary{groupplots}

\usepackage{hyperref}

\pgfplotsset{compat=1.18}

\newcommand{\noi}{\noindent}

\newtheorem{theo}{{{Theorem}}}

\newtheorem{ex}{{ Example}}

\newtheorem{defi}{{Definition}}

\newcommand{\bA}{{\mathbf A}}
\newcommand{\bB}{{\mathbf B}}
\newcommand{\bC}{{\mathbf C}}

\newcommand{\bH}{{\mathbf H}}
\newcommand{\bI}{{\mathbf I}}

\newcommand{\bK}{{\mathbf K}}
\newcommand{\bL}{{\mathbf L}}
\newcommand{\bM}{{\mathbf M}}

\newcommand{\bP}{{\mathbf P}}
\newcommand{\bQ}{{\mathbf Q}}

\newcommand{\bS}{{\mathbf S}}

\newcommand{\bU}{{\mathbf U}}
\newcommand{\bV}{{\mathbf V}}
\newcommand{\bW}{{\mathbf W}}

\newcommand{\bY}{{\mathbf Y}}
\newcommand{\bZ}{{\mathbf Z}}

\newcommand{\system}{{ \mathbf{\Sigma}}}
\newcommand{\bx}{\mathbf{x}}
\newcommand{\by}{\mathbf{y}}
\newcommand{\be}{\mathbf{e}}
\newcommand{\bu}{\mathbf{u}}
\newcommand{\bh}{\mathbf{h}}

\newcommand{\IR}{{\mathbb R}}

\newcommand{\tP}{\widetilde{\bP}}
\newcommand{\tQ}{\widetilde{\bQ}}
\newcommand{\tU}{\widetilde{\mathbf{U}}}
\newcommand{\tL}{\widetilde{\mathbf{L}}}
\newcommand{\tM}{\widetilde{\mathbf{\mathbb{M}}}}

\newcommand{\tilL}{\widetilde{\mathbf{\mathbb{H}}}}
\newcommand{\tg}{\widetilde{\mathbb{{g}}}}
\newcommand{\tilh}{\widetilde{\mathbb{{h}}}}

\newcommand{\M}{\widetilde{\mathbb M}}
\newcommand{\K}{\widetilde{\mathbb K}}

\newcommand{\qbt}{\textsf{LQO-QBT}\xspace}

\IEEEoverridecommandlockouts                              

\title{Data-driven balanced truncation for linear systems with quadratic outputs}

\author{Reetish Padhi$^1$, Ion Victor Gosea$^{2}$, Igor Pontes Duff$^3$ and Serkan Gugercin$^4$
\thanks{Parts of this research were conducted while Padhi was visiting the CSC Group at the Max Planck Institute (MPI) in Magdeburg, Germany. Gugercin's work was supported in part by the
National Science Foundation (NSF), United States, under Grant No. CMMI-2130695.}
\thanks{$^1$R. Padhi and $^4$S. Gugercin is with the Department of Mathematics, Virginia Tech, Blacksburg, VA, USA (e-mail: {reetishp,gugercin}@vt.edu). }
\thanks{$^2$I. V. Gosea and $^3$I. Pontes Duff are with the CSC group, Max Planck Institute of Dynamics of Complex Technical Systems, Magdeburg, Germany (e-mail: {gosea,pontes}@mpi-magdeburg.mpg.de).}}

\begin{document}

\maketitle
\thispagestyle{empty} 
\pagestyle{empty} 

\begin{abstract}
We develop the framework for a non-intrusive, quadrature-based method for approximate balanced truncation (QuadBT) of linear systems with quadratic outputs, thus extending the applicability of QuadBT, which was originally designed for data-driven balanced truncation of standard linear systems with linear outputs only. The new approach makes use of the time-domain and frequency-domain quadrature-based representation of the system's infinite Gramians, only implicitly. We show that by sampling solely the extended impulse responses of the original system and their derivatives (or the corresponding transfer functions), we construct a reduced-order model that mimics the approximation quality of the intrusive (projection-based) balanced truncation. We validate the proposed framework
on a numerical example.
\end{abstract}

\section{Introduction}
\label{sec:introduction}

Linear systems with quadratic output (LQOs) are represented in state-space by
\begin{align} \label{eq:LQO}
   \system_{\textrm{LQO}}: \begin{cases} \dot\bx(t) = \bA\bx(t)+\bB\bu(t),\\
    \by(t) =\bC\bx(t)+\bx^\top(t)\bM\bx(t),
    \end{cases}
\end{align}
where $\bx(t)\in \IR^n$ is the state variable, $\bu(t)\in\IR$ is the control input, $\by(t)$ is the  output. For simplicity, we consider, for the moment,  the single-input single-output (SISO) case; extensions to multiple-input multiple-output (MIMO) systems will be discussed in Sec. \ref{sec:MIMO}. The system matrices are $\bA, \bM \in \IR^{n \times n}, \bB, \bC^T \in \IR^{n}$. We assume that the eigenvalues of $\bA$ are in the open left half-plane and thus 
$\system_{\textrm{LQO}}$ is asymptotically stable. The state equation in~\eqref{eq:LQO} is linear in $\bx(t)$, while the output $\by(t)$ has a quadratic form in $\bx(t)$. Without loss of generality, $\bM$ can be taken to be a symmetric matrix since replacing $\bx^\top\bM\bx$ with $\bx^\top\frac{(\bM+\bM^\top)}{2}\bx$ does not change the output of the system. Thus, in what follows, we assume that $\bM$ is a symmetric matrix. 

LQO systems have recently attracted significant attention as they arise in practical scenarios in which one is interested in observing quantities like the variance of a state from a reference level \cite{PulchNarayan19}. This is the case in problems concerning the observation of quantities like energy or power \cite{morDiaHGetal23}. Moreover, as shown in \cite{holicki2025energy}, adding the Hamiltonian (the energy stored in the system) as an additional observed output, an LTI port-Hamiltonian system becomes an LQO system.

For linear time-invariant (LTI) systems, that is, when $\bM= \mathbf{0}$ and $\by(t)= \bC \bx(t)$ in~\eqref{eq:LQO},
the field of model order reduction (MOR) is fairly established; we refer the reader to, e.g., \cite{ACA05,morAntBG20} for a comprehensive overview. {In this paper, we focus on MOR methods related to balanced truncation \cite{moore2003principal}. In particular, we develop a data-driven 
formulation of BT for LQO systems, thus extending the so-called quadrature-based BT method (QuadBT) of \cite{gosea2022data} to LQQ systems.}

To develop the QuadBT framework for LQO systems, we first use the time-domain representations of the Gramians and their quadrature-based formulations. The resulting approximate BT approach requires only the samples of the system's time-domain kernels and their derivatives, thus alleviating the need to explicitly access the state-space matrices. We also discuss how the framework can be applied using frequency-domain data. Once the relevant data are collected and the data-based surrogate matrices are formed, the reduced matrices are computed by keeping the dominant components of the singular value decomposition of a relevant quantity as in BT. We emphasize that quadrature approximation of the Gramians is implicit and forms the theory behind the framework; these quadrature-based approximations are never explicitly computed in practice.

\section{Model reduction approaches for LQO systems}
MOR for LQO systems can be broadly classified into two categories. The first one typically commences with lifting transformations used to embed the LQO structure into another one for which conventional MOR methods can be directly applied. Specifically, \cite{VanBeeumen12} uses lifting to rewrite the SISO LQO system as a SIMO (single-input multiple-output) linear system (with linear output). After lifting, \cite{VanBeeumen12} employs classical techniques on the lifted linear system. The drawback is that the lifted system suffers from the problem of having a large number of outputs, making the method computationally intensive. In \cite{PulchNarayan19}, the LQO system is lifted to a quadratic-bilinear (QB) system, followed by an application of the BT method introduced in \cite{benner2017balanced} for the lifted QB system. {This method does not increase the output dimension; however, it results in a reduced model that is quadratic-bilinear in dynamics, thus the LQO structure is not retained.} 

Instead of lifting to a different class of systems, we work with the original LQO structure and consider approximating it with an LQO ROM given, in state-space, by
\begin{align} \label{eq:ROM}
\begin{cases}
    \dot\bx_r(t)&=\bA_r\bx_r(t)+\bB_r\bu(t)\\
    \by_r(t)&=\bC_r\bx_r(t)+\bx_r^\top(t)\bM\bx_r(t),
\end{cases}
\end{align} where $\bx_r(t)\in \IR^r$ with $r \ll n$ is the reduced state variable and $\by_r(t)\in\IR$ is the  output of ROM; thus the ROM state-space matrices are
$\bA_r \in \IR^{r\times r}$,
$\bB_r \in \IR^{r}$,
$\bC_r \in \IR^{1\times r}$, and 
$\bM_r \in \IR^{r\times r}$.
One can consider applying BT for general nonlinear dynamical systems 
(see, e.g., \cite{7742345},\cite{SCHERPEN1993143}) to~\eqref{eq:LQO} to construct the ROM; however, for truly large-scale problems, 
solving the resulting Hamilton-Bellman equations is rather computationally expensive. Moreover,  due to the state-dependent nonlinear transformation, the LQO structure is lost. Instead,
\cite{benner2021gramians} develops the algebraic Gramians for LQO systems, which are then used in the BT algorithm specifically tailored to LQO systems. Recently, \cite{morPrzPGetall24} extended this approach to the class of descriptor systems, and \cite{balicki2025polynomial} proposed an alternate perspective on the balanced truncation method for linear systems with polynomial outputs. Aside from the BT approaches, interpolatory projection-based approaches were proposed in \cite{morGosA19}. Also, projection-based $\mathcal H_2$ optimal methods for LQO systems have been recently introduced in \cite{reiter2024h2},\cite{reiter2025data}, and \cite{reiter2023what}. Finally, an optimization-based approach on multi-objective minimization was proposed in \cite{holicki2025energy}. So far, all methods mentioned here are intrusive in the sense that they require access to the state-space formulation~\eqref{eq:LQO}. As far as the authors are aware, the only non-intrusive  MOR of LQO systems is provided in \cite{gosea2022data_AAA}, which utilizes frequency-domain data. In the rest of the paper, we will derive a non-intrusive data-driven formulation of BT for LQO systems.  

\section{Balanced truncation for LQO systems}
 The time-domain kernels, algebraic Gramians, and an algorithm for balanced truncation for LQO systems have been introduced in \cite{benner2021gramians}; we recall some of the theory here.
\begin{defi}
     The time-domain generalized kernel functions of the system $\system_{\textrm{LQO}}$~\eqref{eq:LQO} are defined as
\begin{subequations}  \label{eq:LQOkernels}
        \begin{align}
    \bh_1(\zeta)&=\bC e^{\bA\zeta}\bB,~\zeta_1 \geq 0, \label{eq:lin_kernel}\\
    \bh_2(\zeta_1,\zeta_2)&= \bB^{\top} e^{\bA^{\top}\zeta_1}\bM e^{\bA\zeta_2}\bB,~\zeta_1,\zeta_2\geq 0.\label{eq:quad_kernel}   
    \end{align}
    \end{subequations}
\end{defi}
\noi (The choice of $\zeta$ over $t$ or $\tau$ is to avoid notational conflicts with the quadrature points to be introduced later.) While $\bh_1(\zeta)$ is the standard kernel for LTI systems and corresponds to the linear part of the output, the second bivariate kernel $\bh_2(\zeta_1,\zeta_2)$ corresponds to the quadratic component in the output and can be rewritten in the Kronecker product format as $\bh_2(\zeta_1,\zeta_2) = \bK^T(e^{\bA\zeta_1}\bB\otimes e^{\bA\zeta_2}\bB)$ where $\bK \in \IR^{n^2}$ is the vectorization of matrix $\bM$ (column-wise). 
One can easily compute the (partial) derivatives of  $\bh_1$ and  $\bh_2$, e.g.,
  \begin{subequations}\label{eq:h_der}
    \begin{align}
    \frac{d \bh_1}{d \zeta}(\zeta)&=\bC \bA e^{\bA \zeta} \bB,~\mbox{and}~\label{eq:h1_der}\\
    \frac{\partial \bh_2}{\partial \zeta_2}(\zeta_1,\zeta_2)&=\bB^\top e^{\bA^\top \zeta_1} \bM \bA e^{\bA \zeta_2} \bB.\label{eq:h2_der}
\end{align}
\end{subequations}
Since the state-equation in~\eqref{eq:LQO} is linear, the controllability Gramian $\bP \in \IR^{n \times n}$ for the LQO system is the same as in the standard LTI case, i.e., 
    \begin{align}  \label{eq:Pint}
    \bP&=\int_0^{\infty} e^{\bA t} {\bB \bB}^{\top} e^{\bA^{\top} t} d t.
    \end{align}
\begin{defi}
The infinite observability Gramian $\bQ$ of~\eqref{eq:LQO} is defined as $\bQ =\bQ_1+\bQ_2 \in \IR^{n \times n}$ where
    \begin{subequations} \label{eq:Q1Q2int}
        \begin{align} \label{eq:Q1int}
    \bQ_1&=\int_0^{\infty} e^{\bA^{\top} t} \bC^{\top} \bC e^{\bA t} d t,\\
    \bQ_2 & =\int_0^{\infty} e^{\bA^{\top} t} \bM^{\top} \bP \bM e^{\bA t} d t. \label{eq:Q2int}
    \end{align}
    \end{subequations}
\end{defi}

The Gramians $\bP$ and $\bQ$ satisfy the following (linear) Lyapunov matrix equations \cite{benner2021gramians}:
\begin{subequations}
    \begin{align}
    \bA\bP+\bP\bA^\top+\bB\bB^\top&=0,\label{eq:P_gram_eqn}\\
\bA^\top\bQ+\bQ\bA+\bM^\top\bP\bM+\bC^\top\bC&=0.\label{eq:Q_gram_eqn}
\end{align}
\end{subequations}

As in BT for LTI systems, one does not need the Gramians $\bP$
and $\bQ$, but rather their \emph{square-root factors}. 
These square-root factors  $\bU,\bL_1,\bL_2,\bL \in \IR^{n \times n}$ of, respectively, $\bP,\bQ_1,\bQ_2$, and $\bQ$ satisfy
\begin{align}
\bP = \bU \bU^T,~\bQ_1 = \bL_1 \bL_1^T,~\bQ_2 = \bL_2 \bL_2^T,~\bQ = \bL \bL^T.
\end{align}

The square-root factors can be computed directly without forming $\bP$ and $\bQ$, see, e.g., \cite{ACA05}. 
Once the square-root factors are computed, BT for LQO systems~\cite{benner2021gramians} proceeds in the same fashion as its LTI classical counterpart. First, the singular value decomposition (SVD) of the matrix product $\bL^\top\bU$ is computed, and a truncation value $n>r\geq 1$ is chosen based on the decay of the singular values, followed by partitioning the matrices:
\begin{align}\label{eq:SVD_BT}
    \bL^\top\bU=\bZ \bS \bY^\top = \left[\begin{array}{c}
    {\bZ}_1 ~{\bZ}_2
    \end{array}\right]\left[\begin{array}{ll}
    {\bS}_1 & \\
    & {\bS}_2
    \end{array}\right]\left[\begin{array}{l}
    {\bY}_1^{\top} \\
    {\bY}_2^{\top}
    \end{array}\right].
\end{align} The matrices $\bY_1\in\mathbb{R}^{n\times r},\bZ_1\in\mathbb R^{n\times r}$ and $\bS_1\in\mathbb R^{r\times r}$ correspond to the dominant subsystem. The singular values of matrix $\bL^\top\bU$, i.e., 
the diagonal entries of $\bS$,
are called the LQO system's Hankel singular values, as in the LTI case. The BT ROM is then obtained by truncating the system, by neglecting the smallest $n-r$ singular values. These matrices are then used to construct the model reduction bases  $\bW_r=\bL\bZ_1\bS_1^{-1/2} \in \IR^{n \times r}$ and $\bV_r=\bU\bY_1\bS_1^{-1/2}\in \IR^{n \times r}$. The BT-based ROM~\eqref{eq:ROM} is then obtained via projection as
\begin{equation}\label{eq:LQOBT_terms}
    \begin{aligned}
        \mathbf{A}_r&=\bW_r^T \bA \bV_r=\mathbf{S}_1^{-1 / 2} \mathbf{Z}_1^\top\left(\mathbf{L}^\top \mathbf{A} \mathbf{U}\right) \mathbf{Y}_1 \mathbf{S}_1^{-1 / 2},\\
        \mathbf{B}_r&=\bW_r^T \bB=\mathbf{S}_1^{-1 / 2} \mathbf{Z}_1^\top\left(\mathbf{L}^\top \mathbf{B}\right),\\ \mathbf{C}_r&=\bC \bV_r=(\mathbf{C U}) \mathbf{Y}_1 \mathbf{S}_1^{-1 / 2},\\
        \bM_r&=\bV_r^T \bM \bV_r=\mathbf{S}_1^{-1 / 2}\bY^\top_1(\bU^\top\bM\bU)\mathbf{Y}_1 \mathbf{S}_1^{-1 / 2}.
    \end{aligned}
\end{equation}
By construction, $\bW_r^T  \bV_r = \bI_r$.

\section{Quadrature-based BT method}
\label{sec:QuadBt}

The BT method, as formulated above, is inherently intrusive since it requires direct access to internal system dynamics via $\bA,\bB,\bC,\bM$ as shown in formulae in (\ref{eq:LQOBT_terms}). For LTI systems, the QuadBT approach developed in \cite{gosea2022data} addresses this very issue by computing (approximate) BT ROMs solely from data-based quantities, i.e., from the samples of the kernels and their derivatives. In the standard LTI case, this is achieved by \emph{implicitly} approximating the Gramians with numerical quadratures, in either the time or frequency domains \cite{gosea2022data}. We proceed here in a similar manner for the LQO extensions, by \emph{implicitly} approximating the square root factors $\bU$ and $\bL$ with appropriate counterparts. As we will see later on, these approximate roots are never explicitly formed; rather, they are solely used for the analytical derivations.  We refer the reader to 
\cite{morReiGG24,morLilG24,reiter2025data,morRei25} for extensions of QuadBT to other settings. 

Recall the reduced-order terms in (\ref{eq:LQOBT_terms}) and note that matrices $\bS_1,\bY_1,\bZ_1$ are obtained from the SVD of $\bL^\top\bU$. This leaves us with five terms to compute to construct the BT-based ROM, namely $\bL^\top\bU, \mathbf{L}^\top \mathbf{A} \mathbf{U}, \mathbf{L}^\top \mathbf{B}, ~\mathbf{C U}, \bU^\top\bM\bU$. Thus, to obtain a ROM in a non-intrusive fashion, we need to substitute these terms with appropriate data-based quantities. This is precisely what we achieve in the rest of the section. 

\subsection{Implicit quadrature approximations for Gramians}

The first step  is to replace 
the Gramians $\bP$, $\bQ_1$, and $\bQ_2$ with their quadrature-based approximations. 
Using~\eqref{eq:Pint}, 
the reachability Gramian $\bP$ can be approximated via \begin{align} \label{eq:Ptil}
    \bP \approx \tP = \sum_{i=1}^{N_p}\left(\rho_i e^{\bA t_i} \bB\right)\left(\rho_i e^{\bA t_i} \bB\right)^{\top} = \tU\tU^{\top},
\end{align} where 
$t_i$ and $\rho_i^2$, for $1 \leq i \leq N_p$  are, respectively, the quadrature nodes and weights, and 
the approximate square-root factor $\tU$ is  \begin{align} \label{eq:U_approx}
   \tU&= \left[\rho_1 e^{\bA t_1} \bB \ \ \cdots \ \ \rho_{N_p} e^{\bA t_{N_p}} \bB\right]\in\mathbb{R}^{n\times N_p}.
\end{align} 
We emphasize that the quadrature approximation~\eqref{eq:Ptil} will never be formed explicitly, but it will guide the method development.
The next step is to similarly approximate the two components $\bQ_1,\bQ_2$ in~\eqref{eq:Q1Q2int} of the observability Gramians $\bQ$ with quadrature-based Gramians,

\vspace{-1mm}
 {\small
  \begin{equation}
 \begin{aligned}
   \bQ_1 \approx \tQ_1&=\sum_{j=1}^{N_q}\phi_j^2 e^{\bA^{\top} \tau_j} \bC^{\top} \bC e^{\bA\tau_j} =\tL_1 \tL_1^{\top}\\
  \bQ_2 \approx  \tQ_2&=\sum_{j=1}^{N_q}\phi_j^2 e^{\bA^{\top} \tau_j} \bM^{\top} \tU\tU^\top  \bM e^{\bA \tau_j}=\tL_2 \tL_2^\top,
\end{aligned}
\end{equation} }

\noindent where 
$\phi_j^2$ and $\tau_j$ are quadrature weights and nodes, for $1 \leq i \leq N_p$, and 
the approximate square-root factors $\tL_1\in\mathbb{R}^{n\times N_q}$ and $\tL_2\in\mathbb{R}^{n\times N_pN_q}$ are given as
\begin{subequations}  \label{eq:L1L2approx}
    \begin{align}
    \tL_1&=\left[\phi_1 e^{\bA^{\top} \tau_1} \bC^{\top} \ \ \cdots \ \  \phi_{N_q} e^{\bA^{\top} \tau_{N_q}} \bC^{\top}\right],
    \label{eq:L1approx}\\
    \tL_2&=\left[\phi_1 e^{\bA^{\top} \tau_1} \bM^{\top} \tU \ \ \cdots \ \ \phi_{N_q} e^{\bA^{\top} \tau_{N_q}} \bM^{\top} \tU \right]. \label{eq:L2approx}
\end{align}
\end{subequations}
This allows us to construct a quadrature-based approximation to $\bQ$ and its square-root factor as
\begin{align}
 \bQ \approx   \tQ&=\tQ_1 + \tQ_2= \tL_1 \tL_1^\top + \tL_2 \tL_2^\top = \tL\tL^\top, \\
 \mbox{where}~   \tL&=\left[\tL_1 \quad \tL_2\right]\in\mathbb{R}^{n\times N_pN_q}. \label{eq:L_approx}
\end{align}
Using these approximate quantities, we now define the surrogate quantities to those appearing in (\ref{eq:LQOBT_terms}), namely
\begin{subequations} 
\label{eq:tildesur}
    \begin{align}
\tilL=\tL^\top\tU,\quad \tM&=\tL^\top\bA\tU, \label{eq:HtMt} \\ \tilh=\tL^\top\bB,\quad \tg&=\bC\tU \quad \mbox{and} \quad \K=\tU^\top\bM\tU. \label{eq:htgtKt}
\end{align}
\end{subequations}
The quantities in~\eqref{eq:tildesur} are the five matrices needed to compute BT-based ROM. 
Next, we proceed to prove that these matrices can be explicitly written in terms of data, i.e., as products of quadrature weights and appropriate kernels and their derivatives evaluated at specific quadrature nodes, thus forming the basis of data-driven BT for LQO systems. 
\subsection{Replacing intrusive terms with data matrices}
The first matrix that we consider is $\tilL=\tL^{\top}\tU$, which acts as a surrogate for the matrix $\bL^\top \bU$. For classical LTI systems and for equidistant nodes, this matrix has a Hankel structure~\cite{gosea2022data}.
\begin{theo}\label{thm:LQO_L'U}
   For $\tU$ and $\tL$ as in \eqref{eq:U_approx} and \eqref{eq:L_approx}, the matrix $\tilL=\tL^{\top}\tU \in \IR^{(N_q+N_qN_p)\times N_p}$ can be explicitly expressed as
   \begin{align}  \label{eq:Htilde}       \tilL=\left[\begin{array}{c}
            \tilL_1 \\ \tilL_2 
       \end{array}\right]=\left[\begin{array}{c}
            \tL_1^\top\tU\\ \tL_2^\top\tU
       \end{array}\right]
   \end{align} with the entries of $\tilL_1 \in \IR^{N_q \times N_p}$ and $\tilL_2 \in \IR^{N_q N_p \times N_p} $ as
   \begin{subequations} \label{eq:Htildeji}
       \begin{align}  \label{eq:Htilde1ji}
       (\tilL_1)_{j,i}&=\rho_i\phi_j \bh_1(t_i+\tau_j),\\
       (\tilL_2)_{N_q(k-1)+j,i}&=\rho_i\phi_j\rho_k \bh_2(t_k,\tau_j+t_i)  \label{eq:Htilde2ji}
   \end{align}
      \end{subequations}
for $i,k = 1,\ldots,N_p$ and $j=1,\ldots,N_q$, where $\bh_1$ and $\bh_2$ are the kernel functions in~\eqref{eq:LQOkernels}.
\end{theo}
\begin{proof}
The block structure in~\eqref{eq:Htilde} follows directly from the block structure of $\tL$ in~\eqref{eq:L_approx}.
Using~\eqref{eq:U_approx} and~\eqref{eq:L1L2approx},
the elements of the first block $\tilL_1=\tL_1^{\top}\tU$ can be computed as \begin{align*}
    (\tilL_1)_{j,i}&= \be_j^\top \tL_1^{\top}\tU \be_i =  
    (\phi_j\bC e^{\bA\tau_j})(\rho_ie^{\bA t_i}\bB),\\
    &=\phi_j\rho_i 
    \bC e^{\bA(t_i+\tau_j)}\bB=\rho_i\phi_j\bh_1(t_i+\tau_j),
\end{align*} 
for $i = 1,\ldots,N_p$ and $j=1,\ldots,N_q$ where $\be_i$ denotes the $i$th canonical vector. 
Similarly for $\tilL_2=\tL_2^\top\tU$ we obtain,
\begin{align*}
    (\tilL_2)_{N_q(k-1)+j,i}&=(\phi_j\rho_k\bB^\top e^{\bA^\top t_k} \bM e^{\bA\tau_j}) (\rho_i e^{\bA t_i} \bB)\\
    &=\rho_i \phi_j \rho_k \bh_2(t_k,\tau_j+t_i),
\end{align*} for $i,k = 1,\ldots,N_p$ and $j=1,\ldots,N_q$.

\end{proof}
Theorem~\ref{thm:LQO_L'U} shows that $\tilL$ can be constructed from input-output data using samples of  $\bh_1$ and $\bh_2$. We illustrate the structure of 
$\tilL$ on a simple example.
\begin{ex}
    For $\tU$ and $\tL$ as in~\eqref{eq:U_approx} and~\eqref{eq:L_approx}, choose $N_p=N_q=2$ with unity weights. Then we have
    \begin{align*}
\tilL=\left[\begin{array}{cc}
           \bh_1(t_1+\tau_1) &  \bh_1(t_2+\tau_1) \\
           \bh_1(t_1+\tau_2) & \bh_1(t_2+\tau_2) \\
           \bh_2(t_1,t_1+\tau_1) &  \bh_2(t_1,t_2+\tau_1)\\
           \bh_2(t_1,t_1+\tau_2)& \bh_2(t_1,t_2+\tau_2) \\
           \bh_2(t_2,t_1+\tau_1) &  \bh_2(t_2,t_2+\tau_1) \\
           \bh_2(t_2,t_1+\tau_2) & \bh_2(t_2,t_2+\tau_2)\\
        \end{array}\right] \in \IR^{6 \times 2}.
    \end{align*}
\end{ex}

\begin{theo}\label{thm:LQO_LAU}
    For $\tU$ and $\tL$ as in \eqref{eq:U_approx} and~\eqref{eq:L_approx}, the matrix $\tilde{\mathbb{M}}=\tL^{\top} \bA \tU \in \IR^{(N_q+N_qN_p)\times N_p}$ can be expressed as,
    \begin{equation} \label{eq:LQO_L'AU}
    \tilde{\mathbb{M}}=\left[\begin{array}{cc}
         \tilde{\mathbb{M}}_1 \\ \tilde{\mathbb{M}}_2
    \end{array}\right]=\left[\begin{array}{cc}
         \tL_1^\top \bA \tU \\ \tL_2^\top \bA \tU 
    \end{array}\right]
    \end{equation} with the entries of $\tilde{\mathbb{M}}_1$ and $\tilde{\mathbb{M}}_2$ given by, \begin{subequations}\label{eq:Mtildeji}
        \begin{align}
        (\tilde{\mathbb{M}}_1)_{j,i}&=\rho_i\phi_j\frac{d\bh_1}{d \zeta}(t_i+\tau_j)\label{eq:Mtilde1ij}\\
        (\tilde{\mathbb{M}}_2)_{N_q(k-1)+j,i}&=\rho_i\phi_j\rho_k\frac{\partial \bh_2}{\partial \zeta_2}(t_k,\tau_j+t_i)\label{eq:Mtilde2ij}
    \end{align}
    \end{subequations}
   for $i,k = 1,\ldots,N_p$ and $j=1,\ldots,N_q$ where the terms $\frac{d \bh_1}{d\zeta},~\frac{\partial \bh_2}{\partial\zeta_2}$ denote the partial derivatives of $\bh_1$ and $\bh_2$ as shown in \eqref{eq:h_der}.
\end{theo}
\begin{proof} The block structure in~\eqref{eq:LQO_L'AU} directly follows from the block structure of $\tL$ in~\eqref{eq:L_approx}. The matrix $\tM=\tL^{\top}\bA\tU$ has two terms $\tM_1$ and $\tM_2$. Using~\eqref{eq:U_approx} and~\eqref{eq:L1L2approx},
the elements of the first block $\tM_1=\tL_1^{\top}\bA\tU$ can be computed as,\begin{align*}
    (\tM_1)_{j,i}&=\be_j^\top\tL_1^{\top}\bA\tU\be_i=(\phi_j\bC e^{\bA\tau_j})\bA(\rho_ie^{\bA t_i}\bB),\\
    &=\phi_j\rho_i 
    \bC \bA e^{\bA(t_i+\tau_j)}\bB=\rho_i\phi_j\frac{d \bh_1}{d \zeta}(t_i+\tau_j),
\end{align*} for $i = 1,\ldots,N_p$ and $j=1,\ldots,N_q$ where $\be_i$ denotes the $i$th canonical vector. Similarly for $\tM_2=\tL_2^\top\bA\tU$ we get, 
\begin{align*}
    (\tM_2)_{N_q(k-1)+j,i}&=(\phi_j\rho_k\bB^\top e^{\bA^\top t_k} \bM \bA e^{\bA\tau_j}) (\rho_i e^{\bA t_i} \bB)\\
    &=\rho_i \phi_j \rho_k \frac{\partial \bh_2}{\partial \zeta_2}(t_k,\tau_j+t_i).
\end{align*} where $1\leq i,k \leq N_p$, $1\leq j \leq N_q$.
\end{proof}

\begin{theo}\label{thm:LQO_others}
    Let $\tU$ and $\tL$ be as in \eqref{eq:U_approx} and \eqref{eq:L_approx}, then the matrices $\tilh=\tL^\top\bB=, ~\tg=\bC\tU$ and $\K=\tU^\top\bM\tU$ can be expressed in terms of the kernels (refer \eqref{eq:lin_kernel} and \eqref{eq:quad_kernel}),
        \begin{align} \label{eq:others}
        \tilh=\left[\begin{array}{c}
             \tilh_1 \\ \tilh_2
\end{array}\right]\in\mathbb{R}^{(N_p+N_qN_p)},~\tg\in\mathbb{R}^{1\times N_p},~ \K\in\mathbb{R}^{N_p\times N_p},
    \end{align}
     with the entries of these matrices given by
     \begin{align*}
        (\tilh_1)_{j}&=\phi_j \bh_1(\tau_j) & (\tilh_2)_{(k-1)N_q+j}&=\phi_j\rho_k\bh_2(t_k,\tau_j)\\
        (\tg)_i&=\rho_i\bh_1(t_i)&(\K)_{i,k}&=\rho_i\rho_k\bh_2(t_i,t_k)
    \end{align*} where $1\leq i,k\leq N_p, 1\leq j\leq N_q$.

    \begin{proof}
        The proof follows similarly to that of Theorem \ref{thm:LQO_L'U}. For $\tilh$, we consider the blocks $\tilh_1$ and $\tilh_2$ whose elements can be expressed in terms of samples of the time domain kernels. The result follows similarly for $\tg$ and $\K$.
    \end{proof}
\end{theo}

\begin{algorithm}[H]
\caption{QBT for LQO systems}\label{alg:QuadBT_LQO_full}
\textbf{Input:} Samples of $\bh_1(\tau_1)$ and $\bh_2(\tau_1,\tau_2)$ and their derivatives at the quadrature nodes $t_1,\dots,t_{N_p},\tau_1,\dots,\tau_{N_q}$, plus quadrature weights $\rho_1,\dots,\rho_{N_p},\phi_1,\dots,\phi_{N_q}$. \\  
\textbf{Output:} The reduced-order LQO system matrices $\bA_r \in \IR^{r \times r}, \bB_r  \in \IR^{r}, \bC_r \in \IR^{1 \times r}$, and $\bM_r  \in \IR^{r \times r}$.
\begin{enumerate}
    \item Construct data matrices $\tilL, \M, \tilh, \tg, \K$ using appropriate quadrature nodes $\rho_i,\phi_j$ as shown in \eqref{eq:Htildeji},~\eqref{eq:Mtildeji} and \eqref{eq:others}.
    \item Compute SVD of $\tilL$, choose an appropriate order $r$ depending on the decay of the singular values of $\tilL$ (while ensuring that $r\leq N_p$) and partition the SVD:
    $$
\tilL=\left[\begin{array}{cc}
    \tilde{\bZ}_1 & \tilde{\bZ}_2
    \end{array}\right]\left[\begin{array}{ll}
    \tilde{\bS}_1 & \\
    & \tilde{\bS}_2
    \end{array}\right]\left[\begin{array}{l}
    \tilde{\bY}_1^{\top} \\
    \tilde{\bY}_2^{\top}
    \end{array}\right],
    $$ 
    where $\tilde\bZ_1\in\mathbb{R}^{(N_qN_p+N_q)\times r}, \tilde\bS_1\in\mathbb{R}^{r\times r}$,$\tilde\bY_1^\top \in\mathbb{R}^{r\times N_p}.$
    \item Construct the data-based BT ROM:
    \begin{align*}
        \bA_r&=\tilde{\bS}_1^{-1 / 2} \tilde{\bZ}_1^{\top} \tM \tilde{\bY}_1 \tilde{\bS}_1^{-1 / 2}, & \bB_r&=\tilde{\bS}_1^{-1 / 2} \tilde{\bZ}_1^{\top} \tilh,\\
        \bC_r&=\tg^{\top} \tilde{\bY}_1^{\top} \tilde{\bS}_1^{-1 / 2}, & \bM_r&=\tilde{\bS}_1^{1 / 2} \tilde{\bY}_1^{\top} \K \tilde{\bY}_1 \tilde{\bS}_1^{-1 / 2}.
    \end{align*}
\end{enumerate}
\end{algorithm}

Theorems~\ref{thm:LQO_L'U},~\ref{thm:LQO_LAU} and \ref{thm:LQO_others}
show how to replace the intrusive terms appearing in the standard BT LQO algorithm by equivalent data-based quantities. Thus, we are ready to propose the data-driven BT for LQO systems. Due to the implicit quadrature-based formulation behind it, we call it Quad BT for LQO systems and denote it by \qbt. The resulting numerical method is sketched in Algorithm~\ref{alg:QuadBT_LQO_full}.

\qbt~as outlined in Algorithm~\ref{alg:QuadBT_LQO_full} only requires access to input-output data in the form of samples of the kernel functions $\bh_1(\zeta)$ and $\bh_2(\zeta_1,\zeta_2)$ and their (partial) derivatives at the selected quadrature nodes; thus extending the time-domain QuadBT~\cite{gosea2022data} from the LTI case to the LQO case. The time-domain QuadBT~\cite{gosea2022data} for the LTI case required sampling $\bh_1(\zeta)$ and its derivatives. Here in the LQO case, the quadratic output leads to the bivariate kernel $\bh_2(\zeta_1,\zeta_2)$, and it naturally appears in the data-driven BT formulation. As Algorithm~\ref{alg:QuadBT_LQO_full} shows, quadrature-based Gramians and their square-root factors are never explicitly formed and are not needed. They are used implicitly to derive analytically what input/output data need to be collected. Unlike the fully linear case, 
sampling the bivariate kernel $\bh_2(\zeta_1,\zeta_2)$ in an experimental setting is not a fully resolved topic and definitely requires further investigation. There have been some attempts to sample/estimate the kernels for nonlinear systems, refer to \cite{rebillat2011identification,rebillat2018comparison} for more details. In future work, we will investigate this theory for LQO systems to devise experimental ways to measure the required data. 

\section{Further considerations for \qbt}

We now focus on various aspects of \qbt such as frequency domain formulation, MIMO implementation, and quadrature selection.
\vspace{-1mm}

\subsection{Frequency domain formulation}
The theorems and algorithm presented so far can be extended to the frequency domain as well. For example, $\bQ_2$ in~\eqref{eq:Q2int} can be defined in the frequency domain  as 
{\small \begin{align} \label{eq:Qtil}
    \bQ_2 &= \frac{1}{2\pi}\int_{-\infty}^{\infty} (i\omega \bI-\bA^\top)^{-1} {\bM\bP\bM} (i\omega \bI-\bA)^{-1} d \omega,
\end{align}} 
and then can be approximated via numerical quadrature as 
{\small \begin{align} \label{eq:Qtil} \tQ_2&=\sum_{j=1}^{N_q}\phi_j^2(i\mathbf s_j \bI-\bA^\top)^{-1} \bM\tU \tU^\top \bM(i\mathbf{s}_j \bI-\bA)^{-1},
\end{align}} 
\noindent where $\mathbf s_j$ and $\phi_j^2$, for $1 \leq j \leq N_q$  are, respectively, the quadrature nodes and weights and $\tU$ is the frequency domain approximate square root factor of $\tP$. One can similarly construct square root factors $\tL_2$ for the $\bQ_2$. Repeating this for $\bP$ and $\bQ_1$, we get an expression for frequency-based factors $\tL$ and $\tU$. This gives us a variant of Algorithm.~\ref{alg:QuadBT_LQO_full} with the elements of the matrices $\tilL, \M, \tilh, \tg, \K$ composed of samples of the transfer functions (Laplace transform of $\bh_1$ and $\bh_2$) and derivatives of the transfer function. For example, the entries of the matrix $\tilL$ in~\eqref{eq:Htilde} are now given by \begin{align*}
    (\tilL_1)_{k,j}&= -\phi_k \rho_j \frac{\bH_1(i s_k)-\bH_1(i\theta_j)}{is_k-i\theta_j},\\
    (\tilL_2)_{N_q(k-1)+j,l}&= -\phi_j  \rho_k \rho_l \frac{\bH_2(i\theta_l,is_j)-\bH_2(i\theta_l,i\theta_k)}{is_j -i \theta_k},
\end{align*} where $\theta_{l,k}$ and $\rho_{l,k}$ for $1\leq l,k\leq N_p$ are the quadrature nodes and weights used in the approximation of $\bP$, and $\bH_1$ and $\bH_2$ are the Laplace transforms of $\bh_1$ and $\bh_2$ given by \begin{align*}
    \bH_1(s)&:= \bC (s\bI-\bA)^{-1}\bB,\\ 
   \bH_2(s_1,s_2)&:= \bB^\top (s_1\bI-\bA^\top)^{-1}\bM(s_2\bI-\bA)^{-1}\bB.
\end{align*}
We note that $\tilL$ is now a diagonally-scaled Loewner matrix, similar to those appearing in the frequency formulation of data-driven balancing for LTI systems \cite{gosea2022data}. The entries of the remaining matrices follow similarly and are omitted due to page constraints.

\subsection{Extension to the MIMO case}
\label{sec:MIMO}

Denote with $m$ the number of inputs and with $p$ the number of outputs of the MIMO LQO system. Then, in~\eqref{eq:LQO},
$\bB \in \IR^{n \times m}$
and $\bu(t) \in \IR^m$; and the output equation reads
\begin{equation}\label{eq:MIMO_LQO}
    \begin{aligned}
    \left[\begin{array}{c}
         \by_1(t)\\
          \vdots\\
          \by_p(t)
    \end{array}\right]&=\bC\bx(t)+\left[\begin{array}{c}
         \bx^\top(t)\bM_1\bx(t)\\
          \vdots\\
          \bx^\top(t)\bM_p\bx(t)
    \end{array}\right],
    \end{aligned}
\end{equation}
where $\by(t) \in\mathbb{R}^p, \ \bM_i\in\mathbb{R}^{n\times n}$ for $1\leq i\leq p$, $\bC \in \IR^{p \times n}, \bB \in \IR^{n \times m}$.
Theorems \ref{thm:LQO_L'U}, \ref{thm:LQO_LAU}, and \ref{thm:LQO_others} hold in that case as well, where the time-domain kernels of the system are now matrices of dimension $p\times m$. The data matrices $\tilL,\M,\tg,\tilh$ and $\K$ will have blocks as their entries. For each matrix $\bM_i$, an appropriate data-based matrix $\K_i$ must be constructed using Theorem~\ref{thm:LQO_others}. Then, one can proceed with those (larger) matrices without any changes applied to Algorithm~\ref{alg:QuadBT_LQO_full}. One can use tangential directions (sketching) \cite{morAntBG20} to lower the complexity. Additionally, one can use randomized-SVD to compute the SVD of $\tilL$. 

\subsection{Quadrature points selection}
The choice of quadrature points influences the success of the method. The quadrature points are typically chosen so as to span a rich time or frequency interval (with respect to the first kernel $\bh_1$). Additionally, as expected, the more nodes that are taken, the better the approximation is. In our numerical examples, we restrict the experiments to the choice of logarithmically spaced nodes with the trapezoid rule for simplicity. Even though this simple choice performed very effectively in our numerical experiments, more sophisticated quadrature schemes, such as the Clenshaw–Curtis quadrature rule, could be used~\cite{gosea2022data}. Using these more sophisticated quadrature rules might be a way to reduce the number of quadrature points required to obtain a high-fidelity data-driven BT ROM. 
\section{Numerical experiments}
\label{sec:numerics}
We now provide a proof of concept for our data-driven \qbt framework by comparing it with the intrusive projection-based BT method for LQO systems \cite{benner2021gramians}. All source and data files implemented in MATLAB$^@$, 23.2.0.2428915 (R2023b) Update 4.
and are available at \cite{Padhi_QuadBT_for_LQO_2025}.

 We use the International Space Station (ISS) benchmark (referred to as [iss1r]), modeling the 1R component of the ISS~\cite{GugA01}. The iss1r model is an LTI system of dimension $n=270$ with $m=3$  inputs, and $p=3$ outputs. We modify the system by adding a tridiagonal matrix $\bM=\operatorname{tridiag}(1,2,1)$ in the observed output, and by selecting only the first input and first output to obtain a SISO LQO system. We construct reduced-order models of order $r=30$ using both projection-based intrusive BT and data-driven \qbt as outlined in Algorithm~\ref{alg:QuadBT_LQO_full}. 
For \qbt, the necessary
data are obtained numerically
by explicitly sampling the relevant kernel.
The  \qbt ROMs are then computed according to Algorithm~\ref{alg:QuadBT_LQO_full}, which only sees the input-output data.
For \qbt, we implicitly employ a standard trapezoid quadrature scheme with $N_p=N_q=800$, logarithmically equi-spaced nodes in the time interval $[10^{-1},10^2]$ s. In Fig.~\ref{fig:HSV_iss}, we plot the leading 30 \emph{true} normalized Hankel singular values of the full model and the data-driven ones obtained via \qbt ROM, illustrating that using only the input-output data, the \qbt framework provides a high-fidelity approximation to the true Hankel singular values. To test the accuracy of the projection-based BT ROM and the data-driven \qbt ROM, we simulate the FOM and both ROMS with the input  $\bu(t)=5(\cos(5\pi t)+\sin(12\pi t)e^{-0.4t})$. 

\begin{figure}[h!]
    \centering
        \begin{tikzpicture}
        \begin{axis}[
            xlabel={Index of the singular value},
            ylabel={$\sigma_i/\sigma_1$},
            title={Normalized Hankel singular values},
            ymode=log,
            xmin=1,
            xmax=28,
            ymin=0.000001,
            ymax=1,
            grid=none,
            legend style={at={(0.95,0.95)}, anchor=north east, font=\small },
            width=8cm,
            height=4cm,
            axis x line=bottom,
            axis y line=left,
            enlarge x limits=0.1,
            enlarge y limits=0.1,
            axis lines=box,
        ]
        
        \addplot[
            color=blue,
            mark=o,
            mark options={fill=blue},
            line width=1pt,
        ] table [ y=HSV_f,x=Index, col sep=comma] {Data/HSV_f.csv};
        \addlegendentry{FOM}
        
        \addplot[
            color=red,
            mark=asterisk,
            line width=1pt,
        ] table [y=HSV_f, x=Index, col sep=comma] {Data/HSV_r.csv};
        \addlegendentry{QBT ROM}
        
        \end{axis}
    \end{tikzpicture}
    \caption{Hankel singular values (normalized) of the SISO [iss1r] system and the reduced systems obtained by BT and \qbt}
    \label{fig:HSV_iss}
\end{figure}
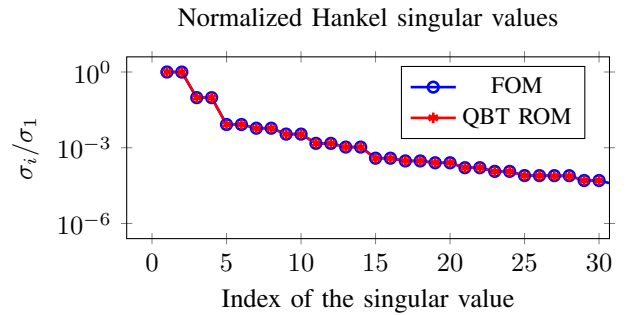

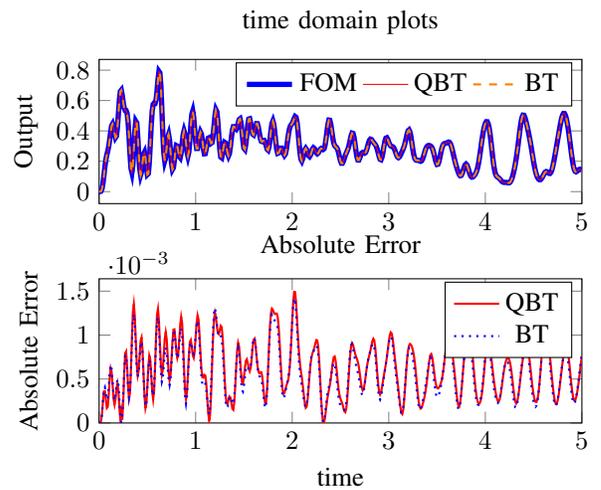
\begin{figure}[h!]
    \centering
    \begin{tikzpicture}

    \begin{groupplot}[
        group style={
            group size=1 by 2, 
            vertical sep=1cm, 
        },
        width=8cm, 
        height=3.5cm,
        ]
        
        \nextgroupplot[
            title={time domain plots}, xmin=0, xmax=5,
            ylabel={Output},
            legend columns=-1
        ]
        
        \addplot[blue, thick, line width=2pt, solid] table [y=FOM_Output, x=Time, col sep=comma] {Data/error_plot_iss.csv};
        \addlegendentry{FOM}

        \addplot[red, thin, mark=., mark size=0.5pt] table [y=QBT_Output, x=Time, col sep=comma] {Data/error_plot_iss.csv};
        \addlegendentry{QBT}
        
        \addplot[orange, thick, dashed] table [y=BT_Output, x=Time, col sep=comma] {Data/error_plot_iss.csv};
        \addlegendentry{BT}

        \nextgroupplot[
            title={Absolute Error},
            ylabel={Absolute Error},
            ymin=0, 
            xmin=0, xmax=5,
            xlabel={time},
        ]
        
        \addplot[red, thick, solid] table [y=Abs_Error_QBT, x=Time, col sep=comma] {Data/error_plot_iss.csv};
        \addlegendentry{QBT}
        
        \addplot[blue, thick, dotted] table [y=Abs_Error_BT, x=Time, col sep=comma] {Data/error_plot_iss.csv};
        \addlegendentry{BT}

    \end{groupplot}
\end{tikzpicture}
    \caption{Time evolution of the SISO [iss1r] system and reduced models obtained by \qbt and BT for chosen input $u(t)$}
    \label{fig:LQO_iss1r_plot}
\end{figure}
Fig.~\ref{fig:LQO_iss1r_plot} shows the observed outputs (top plot) together with the absolute approximation errors (bottom plot). We observe that the data-driven \qbt ROM almost exactly replicates the high-fidelity of the projection-based BT ROM. Next,  we vary the number of quadrature nodes and study its effect on \qbt performance by comparing it with BT using the $\mathcal{H}_2$ measure for LQO systems. Fig.~\ref{fig:h2_trunc} 
shows the $\mathcal{H}_2$ error between the ROMs and full-order model as a function of the number of quadrature nodes with $N_p=N_q$. The blue line serves as a reference denoting the $\mathcal{H}_2$ error of BT for $r=30$. The figure shows that
for this example \qbt reaches intrusive BT accuracy after $N_p=N_q=800$. In Fig.~\ref{fig:h2_r}, we display the $\mathcal{H}_2$ errors for  BT and \qbt as a function of the reduced order $r$. For every $r$ we use logarithmically spaced nodes in $[10^{-1},10^2]$ for \qbt with $N_p=N_q=800$. Naturally, the $\mathcal{H}_2$ error decreases as $r$ increases both for BT and \qbt. More importantly, 
\qbt very closely replicates the performance of projection-based BT. Indeed, for $r=26$, it even outperforms BT; but this is not our goal. Our goal is to replicate the accuracy of intrusive BT from data.

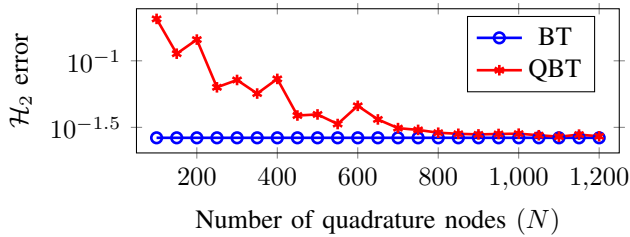
\begin{figure}[h!]
    \centering
        \begin{tikzpicture}
        \begin{axis}[
            xlabel={Number of quadrature nodes $(N)$},
            ylabel={$\mathcal H_2$ error},
            ymode=log,
            xmin=150,
            xmax=1150,
            ymin=0.025,
            ymax=0.2,
            grid=none,
            legend style={at={(0.95,0.95)}, anchor=north east},
            axis lines=box,
            width=8cm,
            height=3.5cm,
            axis x line=bottom,
            axis y line=left,
            enlarge x limits=0.1,
            enlarge y limits=0.1,
            axis lines=box,
        ]
        
        \addplot[
            color=blue,
            mark=o,
            mark options={fill=blue},
            line width=1pt,
        ] table [x=N, y=BT_Error, col sep=comma] {Data/H2_vs_N.csv};
        \addlegendentry{BT}
        
        \addplot[
            color=red,
            mark=asterisk,
            line width=1pt,
        ] table [x=N, y=QBT_Error, col sep=comma] {Data/H2_vs_N.csv};
        \addlegendentry{QBT}
        
        \end{axis}
    \end{tikzpicture}
    \caption{Variation of $\mathcal H_2$ error with number of quadrature nodes ($N$)}
    \label{fig:h2_trunc}
\end{figure}

\begin{figure}[h!]
    \centering
    \begin{tikzpicture}
        \begin{axis}[
            xlabel={Reduced Order $(r)$},
            ylabel={$\mathcal H_2$ error},
            ymode=log,
            xmin=10,
            xmax=60,
            ymin=0.01,
            ymax=0.5,
            grid=none,
            legend style={at={(0.95,0.95)}, anchor=north east},
            xtick={10,20,30,40,50,60},
            width=8cm,
            height=3.5cm,
            axis x line=bottom,
            axis y line=left,
            enlarge x limits=0.1,
            enlarge y limits=0.1,
            axis lines=box,
        ]
        
        \addplot[
            color=blue,
            mark=o,
            mark options={fill=blue},
            line width=1pt,
        ] table [x=Truncation_Index, y=H2_BT_Error, col sep=comma] {Data/H2_truncation.csv};
        \addlegendentry{BT}
        
        \addplot[
            color=red,
            mark=asterisk,
            line width=1pt,
        ] table [x=Truncation_Index, y=H2_r_Error, col sep=comma] {Data/H2_truncation.csv};
        \addlegendentry{QBT}
        
        \end{axis}
    \end{tikzpicture}
    \caption{Variation of $\mathcal H_2$ error with reduced order ($r$)}
    \label{fig:h2_r}
\end{figure}
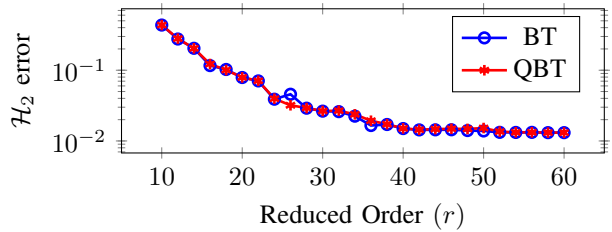

\section{Conclusion and Outlook}
\label{sec:conc}

 We have developed the data-driven QuadBT framework for LQO systems. Using only either time-domain or frequency-domain input/output data, \qbt successfully replicates the performance of projection-based BT. 
Sampling the required data in an experimental setting is not a fully resolved topic and requires further investigation in addition to extensions to other classes of structured nonlinear systems.

\bibliographystyle{abbrv}
\bibliography{ref}

\end{document}